\documentclass[11pt]{article}

\usepackage{amsfonts,amsmath,latexsym,color,epsfig,hyperref,enumitem, amssymb,bbm}

\newcommand{\conv}[1]{\operatorname{conv}#1}

\newtheorem{theorem}{Theorem}[section]
\newtheorem{lemma}{Lemma}[section]

\newtheorem{cor}[lemma]{Corollary}
\newtheorem{prop}[lemma]{Proposition}

\renewcommand{\le}{\leqslant}
\renewcommand{\ge}{\geqslant}
\renewcommand{\leq}{\leqslant}
\renewcommand{\geq}{\geqslant}
\newcommand{\R}{\mathbb R}

\newcommand{\parag}[1]{\vspace{2mm}

\noindent{\bf #1} }

\newcommand{\remark}
      {\medskip\noindent {\bf Remark:\hspace{0em}}}
\newenvironment{proof}
      {\medskip\noindent{\bf Proof:}\hspace{1mm}}
      {\hfill$\Box$\medskip}
\def\qed{\ifvmode\mbox{ }\else\unskip\fi\hskip 1em plus 10fill$\Box$}
\newenvironment{proofof}[1]
      {\medskip\noindent{\bf Proof of #1:}\hspace{1mm}}
      {\hfill$\Box$\medskip}
\input{epsf}

\makeatletter
\def\Ddots{\mathinner{\mkern1mu\raise\p@
\vbox{\kern7\p@\hbox{.}}\mkern2mu
\raise4\p@\hbox{.}\mkern2mu\raise7\p@\hbox{.}\mkern1mu}}
\makeatother

\def\R{\mathbb R}

\def\F{\mathbb F}

\def\F{\mathcal{F}}

\title{\vspace{-0.7cm}Convex polytopes from fewer points}
\author{Cosmin Pohoata \and Dmitrii Zakharov}
\date{}

\begin{document}
\maketitle

\begin{abstract}
Let $ES_{d}(n)$ be the smallest integer such that any set of $ES_{d}(n)$ points in $\mathbb{R}^{d}$ in general
position contains $n$ points in convex position. In 1960, Erd\H{o}s and Szekeres showed that $ES_{2}(n) \geq 2^{n-2} + 1$ holds, and famously conjectured that their construction is optimal. This was nearly settled by Suk in 2017, who showed that $ES_{2}(n) \leq 2^{n+o(n)}$. In this paper, we prove that 
$$ES_{d}(n) = 2^{o(n)}$$
holds for all $d \geq 3$. In particular, this establishes that, in higher dimensions, substantially fewer points are needed in order to ensure the presence of a convex polytope on $n$ vertices, compared to how many are required in the plane.
\end{abstract}

\section{Introduction}
For $d \geq 2$, a set of points $X$ in $\mathbb{R}^{d}$ with $|X| \geq d+1$ is said to be in {\it{general position}} if no $d+1$ points from $X$ lie on the same $(d-1)$-dimensional hyperplane. A set of points $P$ is in {\it{convex position}} if the points from $P$ represent the vertices of a convex polytope. 

In their seminal 1935 paper, Erd\H{o}s and Szekeres \cite{ES35} proved that for every integer $n \geq 3$ there exists a smallest integer $ES_{2}(n)$ such that any set of $ES_{2}(n)$ points in the plane in general position must contain $n$ points in convex position. Their paper contains two different proofs for the existence of $ES_2(n)$, both of which have generated remarkable bodies of work in several directions over the years. Their first argument from \cite{ES35} showed that $ES_{2}(n) \leq R_{4}(5,n)$, and used a quantitative version of Ramsey’s Theorem (see \cite{Ramsey} or \cite[Theorem 4.18]{Jukna}) to obtain a rather poor bound for $ES_{2}(n)$. Here $R_{k}(s,n)$ denotes the standard $2$-color Ramsey number for $k$-uniform hypergraphs, namely the minimum $N$ such that every red-blue
coloring of the unordered $k$-tuples of an $N$-element set contains a red set of size $s$ or a blue set of size $n$, where a set is called red (blue) if all $k$-tuples from this set are red (blue). Their second argument was more geometric in nature and showed a much more refined estimate
$$ES_{2}(n) \leq f(n,n) = {2n-4 \choose n-2} + 1,$$
where $f(k,\ell)$ denotes the smallest integer $N$ such that any planar point set of size $N$ in general position must always contain a $k$-cup or an $\ell$-cap. We refer to \cite{Mat} for a nice exposition of both approaches. In 1960, Erd\H{o}s and Szekeres showed that $ES_{2}(n) \geq 2^{n-2} + 1$ holds, and famously conjectured that their construction is optimal. After a long series of improvements (e.g. \cite{CG98}, \cite{KP98}, \cite{TV98}, \cite{TV06}, \cite{MV16}), this was nearly settled by Suk \cite{Suk17} in 2017, where he showed that $ES_{2}(n) \leq 2^{n+o(n)}$. The best known quantitative bound is due to Holmsen, Mojarrad, Pach, and Tardos \cite{HMPT20}, who optimized (and generalized) the argument from \cite{Suk17} and showed that $ES_{2}(n) \leq 2^{n+O(n^{1/2} \log n)}$. Here and throughout the rest of the paper all asymptotic notation is in the $n \to \infty$ regime. 

Despite a lot of activity around this problem, the higher dimensional story has managed to remain quite mysterious during all this time. As Erd\H{o}s and Szekeres note themselves in \cite{ES35}, the existence of $ES_{d}(n)$ also follows from Ramsey's theorem, applied in the same vein as in their first proof. By Carath\'eodory's theorem \cite{Car07}, it is easy to see that for every $d \geq 2$, any configuration of $d+3$ points in general position in $\mathbb{R}^{d}$ must contain at least $d+2$ points in convex position, so the general estimate $ES_{d}(n) \leq R_{d+2}(d+3,n)$ holds. See also \cite{Dan} or \cite{Grun} for more detailed discussions. Similarly, this only yields a very modest quantitative upper bound for $ES_{d}(n)$, which in fact even becomes worse and worse as the dimension, and thus also the uniformity of the Ramsey number in question, increases. We refer to \cite{CFS10} and \cite{MS18} for the state of the art on these particular hypergraph Ramsey numbers (and several others). 

On the other hand, a simple projection argument, originally due to Valtr \cite{V96} (cf. \cite{MS00}), defies the implicit higher uniformity of the problem in $\mathbb{R}^{d}$. By considering a set of $ES_{d-1}(n)$ points in general position in $\mathbb{R}^{d}$, projecting onto a generic $(d-1)$-dimensional hyperplane, finding a convex subset inside the projection, and then ultimately lifting this set back to get a convex subset in the original configuration, it immediately follows that $ES_{d}(n) \leq ES_{d-1}(n)$ must hold for every $d \geq 3$, i.e. 
\begin{equation} \label{ES}
ES_{d}(n) \leq ES_{d-1}(n) \leq \ldots \leq ES_{2}(n).
\end{equation}
In particular, any upper bound for the two-dimensional problem yields an upper bound for the $ES_{d}(n)$, which means that Suk's theorem automatically implies that $ES_{d}(n) \leq 2^{n + o(n)}$ holds for all $d \geq 2$. The previously best known result for $d \geq 3$ is only a technical refinement of the above projection argument. By projecting onto a generic $(d-1)$-dimensional hyperplane from a fixed point of the configuration rather than from infinity, K\'arolyi \cite{K01} observed that $ES_{d}(n) \leq ES_{d-1}(n-1) + 1$ holds. Nevertheless, this clearly only gives an upper bound of the same (asymptotic) quality as \eqref{ES} for $ES_{d}(n)$ when $d \geq 3$. The question of whether $ES_{3}(n)$ could potentially be asymptotically smaller than $ES_{2}(n)$ has been raised by several researchers in various forms, and even conflicting conjectures have been proposed over the years. See for example \cite[Chapter 3.1, page 33]{Mat} and the beautiful survey \cite{MS00} for nice accounts.

In this paper, we address this problem and confirm that in higher dimensions substantially fewer points are needed in order to ensure the presence of a convex polytope on $n$ vertices, compared to how many are required in the plane. This is already true starting with $d=3$. 

Our main new result is in fact the following subexponential upper bound for the Erd\H{o}s-Szekeres function in $3$-space. 

\begin{theorem} \label{main}
For any $\epsilon > 0 $, there exists $n_{0}(\epsilon)$ such that for every $n \geq n_{0}(\epsilon)$, the following holds: if $X \subset \mathbb{R}^{3}$ is a set of points in general position with $|X| \geq 2^{\epsilon n}$, then $X$ must always contain $n$ points in convex position. In other words,
$$ES_{3}(n) = 2^{o(n)}.$$
\end{theorem}

Together with the inequality chain from \eqref{ES}, Theorem \ref{main} implies that $ES_{d}(n) = 2^{o(n)}$ holds for all $d \geq 3$. Among other things, this disproves the prediction of Morris and Soltan from \cite{MS00}, who conjectured that $ES_{d}(n) = \Omega\left(2^{2n/d}\right)$, and in fact also $ES_{d}(n) = 4ES_{d}(n-d)-3$, should hold for all $d \geq 2$ and $n > \lfloor (3d+1)/2 \rfloor$. 

Our second result is a quantitative version of the so-called positive fraction Erd\H{o}s-Szekeres theorem in $\mathbb{R}^{3}$. We say that a collection of sets $X_1, \ldots, X_n \subset \R^d$ is in convex position if for every $i = 1, \ldots, n$ the convex hulls $\conv(X_i)$ and $\conv(\bigcup_{j \neq i} X_j)$ are disjoint. Note that this is a stronger condition than just to require that for any $x_1 \in X_1, \ldots, x_n \in X_n$ the set $\{x_1, \ldots, x_n\}$ is in convex position.

\begin{theorem} \label{fr}
There exist a sufficiently large positive integer $n_0$ such that for all $n \ge n_0$ the following holds: any set in general position $\mathcal{X} \subset \R^3$ with $|\mathcal{X}| \ge ES_{3}(8n)$ must contain a collection of subsets $X_1, \ldots, X_n \subset \mathcal{X}$ in convex position such that
$$
|X_i| \ge \frac{|\mathcal{X}|}{ES_3(8n)^8}
$$ for every $i = 1, \ldots, n$.
\end{theorem}

It follows from Theorem \ref{main} and Theorem \ref{fr} that for every $n \geq 3$, there is $\epsilon_{n} = (1/2)^{o(n)}$ such that every set $\mathcal{X} \subset \mathbb{R}^{3}$ in general position must contain $n$ subsets $X_{1},\ldots,X_{n} \subset \mathcal{X}$ with $|X_{i}| \geq \epsilon_{n} |\mathcal{X}|$, for all $i = 1,\ldots,n$, and such that for every choice of $x_{1} \in X_1,\ldots,x_{k} \in X_{k}$, the set $\left\{x_1,\ldots,x_{k}\right\}$ is in convex position. By a projection argument in the same style with the one behind \eqref{ES}, it is easy to see that this further implies the following quantitative version of the positive fraction Erd\H{o}s-Szekeres theorem in $\mathbb{R}^{d}$.

\begin{theorem} \label{fr_d}
For every $d \geq 3$ and $n \geq 3$, there exists $\epsilon_{n} = (1/2)^{o(n)}$ such that the following holds: any set $X \subset \mathbb{R}^{d}$ in general position and of size $|\mathcal{X}| \geq ES_{3}(8n)$ must contain a collection of $n$ subsets $X_{1},\ldots,X_{n}$ in convex position with $|X_{i}| \geq \epsilon_{n} |\mathcal{X}|$, for all $i = 1,\ldots,n$.
\end{theorem}

The first such result was established by B\'ar\'any and Valtr in \cite{BV98} in $\mathbb{R}^{2}$ for all sets $\mathcal{X} \subset \mathbb{R}^{2}$ satisfying $|\mathcal{X}| \geq ES_{2}(n)$, with an $\epsilon_n^{-1}$ doubly exponential in $n$. This was later refined by Pach and Solymosi in \cite{PS98}, and then by Por and Valtr in \cite{PV02}, who showed that the planar version of the above statement holds with $\epsilon_n = n \cdot 2^{-32n}$. This is in some sense sharp, because on the other hand it can be shown that there exists a constant $\kappa \approx 1/\sqrt{2}$ and a set $\mathcal{X} \subset \mathbb{R}^{2}$ for which there is no collection of subsets $X_{1},\ldots,X_{n}$ with $|X_{i}| \geq \kappa^{n} |\mathcal{X}|$ for all $i=1,\ldots,n$, and with the required property. See \cite[Section 6.2]{PV02} for more details. In contrast, Theorem \ref{fr_d} shows that for all $d \geq 3$ the positive fraction Erd\H{o}s-Szekeres theorem holds with an $\epsilon_{n}^{-1}$ which is subexponential in $n$.

\section{Preliminaries}

In this section, we collect several results and preliminary lemmas that we will need for the proof of Theorem \ref{main}. 

\smallskip

\parag{Two-dimensional prerequisites.} The first theorem is a well-known result from \cite{ES35}, commonly referred to as the Erd\H{o}s-Szekeres cups-vs-caps theorem. Let $P \subset \mathbb{R}^{2}$ be a set of points in general position, and let $|P| = a$. We say that $P$ is an $a$-cap ($a$-cup) if $P$ is in convex position and its convex hull is bounded from below (above) by a single edge. Equivalently, note that $P$ is a cup if and only if for every point $p \in P$, there is a line $\ell$ containing $p$ such that all $p' \in p$, $p' \neq p$ lie above $\ell$. Similarly $Q \subset \mathbb{R}^{2}$ is a cap if and only if for every $q \in Q$, there is a line $\gamma$ containing $q$ such that all $q' \in Q$, $q' \neq q$ lie below $\gamma$.

\begin{theorem} \label{ESlemma}
Let $a,b \geq 2$ be positive integers, and let $f(a,b)$ be the smallest $N$ such that any set of points $X \subset \mathbb{R}^{2}$ in general position and with $|X| \geq N$ must always contain an $a$-cap or a $b$-cup. Then, 
$$f(a,b) = {a+b - 4 \choose a-2} + 1.$$
\end{theorem}

The next theorem is the planar positive fraction Erd\H{o}s-Szekeres theorem due to P\'or and Valtr \cite{PV02}, discussed above. We record its statement below together with some terminology.

\begin{theorem}\label{lem_pv}
Let $k \geq 3$ and let $X \subset \mathbb{R}^{2}$ be a finite point set in general position such that $|X| \geq 2^{40k}$. Then there is a $k$-element subset $P \subset X$ such that either a $k+1$-cap or a $k+1$-cup, and the regions $T_{1}\ldots,T_{k}$ from the support of $X$ satisfy $|T_{i} \cap X| \geq |X| / 2^{40k}$. In particular, every $k$-tuple obtained by selecting one point from each $T_i \cap X$, $i=1,\ldots,k$ is in a convex  position. 
\end{theorem}

Given a $k+1$-cap or $k+1$-cup $P = \left\{x_1,\ldots,x_{k+1}\right\}$, where the points are sorted from left to right according to some coordinate system, the {\it{support}} of $P$ is the collection of regions $\left\{ T_1,\ldots,T_k\right\}$, where $T_{i}$ is the region outside of $\operatorname{conv}(P)$ is bounded by the segments $x_{i}x_{i+1}$ and by lines $x_{i-1}x_{i}$ and $x_{i+1}x_{i+2}$ (where the indices are taken modulo $k+1$ at the endpoints). Given $X \subset \mathbb{R}^{2}$ and the structure induced by Theorem \ref{lem_pv}, we shall also sometimes call $P$ the {\it{supporting polygon}} of the configuration and the edges $\left\{x_1x_2,\ldots,x_{k}x_{k+1}\right\}$, which are incident to its support, as the {\it{supporting edges}} of $P$. It is perhaps important to also emphasize that the version of Theorem \ref{lem_pv} cited above is not quite the original theorem of P\'or and Valtr from \cite{PV02}, but rather a quick consequence. We refer to \cite{Suk17} for more details about how Theorem \ref{lem_pv} follows from \cite[Theorem 4]{PV02}. 

Both Theorem \ref{ESlemma} and Theorem \ref{lem_pv} played a crucial in Suk's proof from \cite{Suk17}, and, despite their two-dimensional nature, will also play an important role in the proof of Theorem \ref{main}. 

\smallskip

\parag{Cups-vs-caps in $\mathbb{R}^{3}$.} Our next preliminary result is a simple three-dimensional generalization of Theorem \ref{ESlemma}, and which may be of independent interest. The statement requires a little bit of setup.

Given a convex set $C \subset \R^3$, we say that a set $X \subset \R^3$ is {\it{$C$-free}} if for any distinct $x, y \in X$ the line $l = x y$ does not intersect $C$. A set $Y \subset \R^3$ is called a {\it{$C$-cap}} if any point $y \in Y$ does not belong to the set $\conv(C \cup (Y \setminus \{y\}))$.

\begin{prop}\label{lem_es}
Let $P$ be a polytope and let $X \subset \R^3$ be a finite $P$-free set in general position. Let $e(P)$ denote the number of edges of $P$. If for some $a, b \ge 1$ we have $|X| > {a+b - 4 \choose a -2}^{e(P)}$, then either $X$ contains a $P$-cap of size $a$ or a convex set of size $b$.
\end{prop}

\begin{proof}
First observe that for any line $l$ such that $l \cap P = \emptyset$ there exists an edge $e$ of $P$ such that the projections of $l$ and $P$ along $e$ are disjoint. Indeed, let us consider the projection $\pi: \R^3 \rightarrow \R^2$ along $l$. Then $\pi(l)$ is a point and it is disjoint from the polygon $\pi(P)$. So there exists an edge $e'$ of $\pi(P)$ such that the line $(e')$ spanned by $e'$ separates $\pi(l)$ from $\pi(P)$. Let $e$ be an  edge of $P$ in the preimage $\pi^{-1}(e')$; it is then easy to see that $e$ satisfies the desired condition.

Let $e$ be an edge of $P$, denote by $\pi_e$ the projection along $e$.
Define a partial order $\prec_e$ on $X$ as follows: for $x, y \in X$ we have $y \prec_e x$ if and only if $\pi_e(y) \in \conv\left(\pi_e(P) \cup \left\{\pi_e(x)\right\}\right)$. Clearly, $\prec_e$ is a partial order on $X$. The observation above implies that any points $x, y\in X$ are incomparable with respect to $\prec_e$ for at least one edge $e$ of $P$. Dilworth's theorem \cite{RD50} then implies that there exists a set $X' \subset X$ of size at least $|X|^{1/e(P)}$ which is an antichain with respect to the partial order $\prec_e$, for some edge $e$ of $P$. Indeed, if there is no such large antichain with respect to any of the partial orders $\prec_e$, then for every edge $e$ of $P$ there must exist a partition of $X$ into $<|X|^{1/e(P)}$ chains with respect to $\prec_e$. Superimposing these $e(P)$ partitions of $X$ gives a decomposition of $X$ into $<|X|$ sets, where each set is a chain with respect to all partial orders $\prec_e$, $e \in e(P)$. But such a decomposition is impossible: at least two distinct elements $x,y$ of $X$ must fall into the same set, while on the other hand $x$ and $y$ must incomparable with respect to $\prec_e$ for at least one edge $e$ of $P$.

By Theorem \ref{ESlemma}, applied to the projection $\pi_e(X')$ with parameters $a, b$ and an appropriately chosen coordinate system, we conclude that either $\pi_e(X')$ contains a $\pi_e(P)$-cap of size $a$ or a convex set on the plane of size $b$. Lifting either of these sets to $\R^3$ gives us a $P$-cap of size $a$ or a convex set of size $b$ in $X' \subset X$, respectively.
\end{proof}

\smallskip

\parag{Above and below in space.} Let $\pi: \R^3 \rightarrow \R^2$ be the projection onto the first 2 coordinates. For two disjoint line segments $\overline{ab}, \overline{cd} \subset \R^3$ whose projections $\pi(\overline{ab})$ and $\pi(\overline{cd})$ intersect at a point $x \in \R^2$, we say that $\overline{ab}$ {\it{lies above (below)}} $\overline{cd}$ if the third coordinate of the point $\overline{ab}\cap \pi^{-1}(x)$ is larger (smaller) than the third coordinate of the point $\overline{cd}\cap \pi^{-1}(x)$. 

\begin{prop}\label{lem_ab}
For any $k \ge 4$ there exists a number $AB(k)$ such that the following holds for any $N \geq AB(k)$. Let $x_1, \ldots, x_N \in \R^3$ be points in general position such that the projections $\pi(x_i)$, $i=1, \ldots, N$, are consecutive vertices of a convex polygon in $\R^2$. 
Then there is a $k$-element set $S \subset [N]$ such that either for any indices $i < i' < j < j' \in S$ the segment $\overline{x_i x_j}$ is above $\overline{x_{i'} x_{j'}}$ or the same condition holds with `above' replaced by `below'.
\end{prop}
	
\begin{proof}
Consider the following $2$-coloring of the $4$-element subsets of $\left\{x_1,\ldots,x_N\right\}$: for every $4$-tuple $1\leq i<i'<j<j' \leq N$, say $\left\{x_i,x_i',x_j,x_j'\right\}$ is red if the segment $\overline{x_i x_j}$ is above $\overline{x_{i'} x_{j'}}$, and say $\left\{x_i,x_i',x_j,x_j'\right\}$ is blue otherwise. Since the points $x_1,\ldots,x_N$ are in general position, note that the latter happens precisely if and only if the segment $\overline{x_i x_j}$ is below $\overline{x_{i'} x_{j'}}$. 

The conclusion thus follows from Ramsey's theorem (see \cite{Ramsey} or \cite[Theorem 4.18]{Jukna}): any number $AB(k) \geq R_{4}(k,k)$ satisfies the statement. 
\end{proof}	

\begin{prop}\label{pr_ab}
Let $X_1, X_2, X_3, X_4 \subset \R^3$ be pairwise disjoints sets such that $X_1 \cup X_2 \cup X_3 \cup X_4$ is in general position and for any $x_i \in X_i$, $i = 1, \ldots, 4$ the segment $\overline{x_1 x_3}$ is above $\overline{x_2 x_4}$ (in particular, their projections on $\R^2$ intersect). Then convex hulls $\conv(X_1 \cup X_3)$ and $\conv(X_2 \cup X_4)$ are disjoint.
\end{prop}

\begin{proof}
The proof is based on the following classical result known as Kirchberger's theorem \cite{Kir}. See also \cite{Ba} for an excellent exposition.

\begin{theorem}\label{kirch}
Let $A, B \subset \R^d$ be arbitrary non-empty sets such that $\conv(A) \cap \conv(B) \neq \emptyset$. Then there exists a subset $A' \subset A$ and a subset $B' \subset B$ such that $\conv(A') \cap \conv(B') \neq \emptyset$ and $|A'| + |B'| \le d+2$.
\end{theorem}

Now we prove Proposition \ref{pr_ab}. Suppose that convex hulls of sets $X_1 \cup X_3$ and $X_2 \cup X_4$ intersect. Then by Theorem \ref{kirch} we can find sets $A \subset X_1 \cup X_3$ and $B \subset X_2 \cup X_4$ such that $|A|+|B| = 5$ and $\conv(A) \cap \conv(B) \neq \emptyset$.

Let $Y_i = \pi(X_i)$, $i = 1, \ldots, 4$. By assumption, for any $y_i \in Y_i$ the segments $\overline{y_1 y_3}$ and $\overline{y_2 y_4}$ intersect. Then it is easy to see that the collection $Y_1, Y_2, Y_3, Y_4$ is in convex position.
This implies that neither of the sets $A, B$ is fully contained in any set $X_i$, $i=1,\ldots,4$. Without loss of generality, we may assume that $A = \{x_1, x_3\}$, $B = \{x_2, x_4, x_4'\}$ where $x_1 \in X_1$, $x_2 \in X_2$, $x_3 \in X_3$ and $x_4, x_4' \in X_4$. By assumption, both segments $\overline{x_2 x_4}$ and $\overline{x_2 x_4'}$ are below $\overline{x_1 x_3}$. This means that the segment $\overline{x_1 x_3}$ and the triangle $\conv(x_2, x_4, x_4')$ are disjoint. But this contradicts the assumption that $\conv(A) \cap \conv(B) \neq \emptyset$.
\end{proof}

Combining these two propositions we obtain:

\begin{cor}\label{cor_ab}
Let $X \subset \R^3$ be a set of points in general position such that the projection $\pi(X)$ is  in convex position. If $N \ge AB(k)$ then there are points $x_1, \ldots, x_k \in X$ such that $\pi(x_1), \ldots, \pi(x_k)$ are consecutive vertices of a convex polygon on the plane and for any $1 \le a \le b \le c \le k$ the convex hulls of the sets 
\begin{align*}
\{x_{1}, \ldots, x_{a-1}\} \cup \{x_{b}, \ldots, x_{c-1}\}\ \ \ \text{and}\ \ \
\{x_{a}, \ldots, x_{b-1}\} \cup \{x_{c}, \ldots, x_{k}\}
\end{align*}
are disjoint.
\end{cor}

As a sidenote, we believe that finding the smallest value $AB(k)$ for which the conclusion of Proposition \ref{lem_ab} holds for all $N \geq AB(k)$ might be an interesting problem for its own sake. For example, without too much effort, one can readily note that the $2$-coloring from the proof of Proposition \ref{lem_ab} is semialgebraic and of low complexity, which means that the improved quantitative bounds for semi-algebraic Ramsey numbers (e.g. \cite{Suk14}) immediately yield better information about $AB(k)$ than the proof of Proposition \ref{lem_ab} does. Nevertheless, such improvements only seem to have a rather immaterial effect on our $o(n)$ term in Theorem \ref{main}, see also the remark at the end of Section 3 for more details. 

\smallskip

\parag{$2$-Separability.} 
We call a collection of sets $X_1, \ldots, X_k \subset \R^3$ {\it $2$-separated} if for any set of indices $i, j, i', j' \in [k]$ such that $\{i, j\} \cap \{i', j'\} = \emptyset$ we have $$
\conv(X_i \cup X_j) \cap \conv(X_{i'} \cup X_{j'}) = \emptyset.
$$

\begin{prop}\label{lem_sep}
Let $X_1, \ldots, X_k \subset \R^3$ be finite pairwise disjoint sets of size at least $2^{k^3}$ such that $X_1 \cup \ldots \cup X_k$ is in general position. Then, there exist $Y_i \subset X_i$ such that $|Y_i| \ge 2^{-k^3} |X_i|$ for every $i=1,\ldots,k,$ and the collection $Y_1, \ldots, Y_k$ is $2$-separated.
\end{prop}

The proof of Theorem \ref{lem_sep} rests upon the observation that for every $1 \leq i_1 < i_2 < i_3 < i_4 \leq k$, there exist subsets $Y_{i_j} \subset X_{i_j}$ with $|Y_{i_j}| \geq |X_{i_{j}}|/2$ for all $j =1,\ldots,4$, and such that the convex hulls of sets $Y_{i_1} \cup Y_{i_2}$ and $Y_{i_3} \cup Y_{i_4}$ are disjoint. This in turn follows from the following consequence of the so-called ham sandwich theorem from topology, which was originally conjectured by Steinhaus, proved by Banach in 1938, and subsequently generalized by Stone and Tukey in 1942. See for example \cite{Ham} and the references therein. 

For a hyperplane $H \subset \R^d$ we denote by $H^+$ and $H^-$ the two closed half-spaces with boundary $H$. Note that in order for the half-spaces $H^+$ and $H^-$ to be properly defined one also has to fix an orientation on $H$: otherwise, there will be no way to distinguish between $H^+$ and $H^-$. So whenever we talk about half-spaces corresponding to a given hyperplane $H$ we implicitly assume that $H$ is oriented.

\begin{lemma} \label{BU}
Let $d$ and $r$ be integers such that $1 \le r \le d$. Let $X_1, \ldots, X_{d+1} \subset \R^d$ be arbitrary finite sets. Then there exists a hyperplane $H$ such that the closed half-space $H^+$ intersects the sets $X_1, \ldots, X_r$ in at least half of the elements and the closed half-space $H^-$ intersects the sets $X_{r+1},\ldots, X_{d+1}$ in at least half of the elements.
\end{lemma}

\begin{proof}
By the discrete version of the ham sandwich theorem \cite[Theorem 3.1.2]{MatBU}, there exists a hyperplane $H$ such that for any $i = 1, \ldots, d$ we have $|H^+ \cap X_i|, |H^- \cap X_i| \ge |X_i|/2$. For the last set $X_{d+1}$ we have either $|H^+ \cap X_{d+1}| \ge |X_{d+1}|/2$ or $|H^- \cap X_{d+1}| \ge |X_{d+1}|/2$, so after choosing an appropriate orientation of $H$ we obtain the claim.
\end{proof}

\begin{proofof}{Proposition \ref{lem_sep}} By Lemma \ref{BU}, applied in $\mathbb{R}^{3}$ and with $r=2$, it follows that for every $1 \leq i_1 < i_2 < i_3 < i_4 \leq k$, there exist subsets $Y_{i_j} \subset X_{i_j}$ with $|Y_{i_j}| \geq |X_{i_{j}}|/2$ for all $j =1,\ldots,4$, and such that the convex hulls of sets  $Y_{i_1} \cup Y_{i_2}$, $Y_{i_3} \cup Y_{i_4}$ are disjoint. Indeed, this is because one can take $Y_{i_1}$ and $Y_{i_2}$ to be the subsets of $X_{i_1}$ and  $X_{i_2}$ that are in $H^{+}$ and $Y_{i_3} \subset X_{i_3}$ and $Y_{i_4} \subset X_{i_4}$ to be the subsets in $H^{-}$. 
Since the union $Y_{i_1} \cup Y_{i_2} \cup Y_{i_3} \cup Y_{i_4}$ is in general position we can slightly perturb the hyperplane $H$ to ensure that sets $Y_{i_j}$, $j=1, \ldots, 4$ are disjoint from $H$ and are still contained in the respective half-spaces.
Clearly, in this case we must have 
$$\conv(Y_{i_1} \cup Y_{i_2}) \cap \conv(Y_{i_3} \cup Y_{i_4}) \subset H^+ \cap H^- = H,$$
and it follows that the convex hulls are indeed disjoint.

We apply this fact repeatedly, in stages, as follows. 
Label elements of ${[k] \choose 4}$ by numbers from 1 to ${k \choose 4}$ arbitrarily. 
At stage $0$, we have the initial list of (original) sets
$$X_{1}^{(0)}:=X_1,\ldots,X_{k}^{(0)}:=X_k,$$
which we will be updating from step to step.  


For every $r = 1, \ldots,{k \choose 4}$, the list of sets at the end of stage $r$ will consist of $k-4$ of the sets from the list at stage $r-1$ together with $4$ new sets corresponding to the $r$-th $4$-tuple in ${[k] \choose 4 }$. More precisely, if $\ell_1<\ell_2<\ell_3<\ell_4$ is the $r$-th $4$-tuple in ${[k] \choose 4 }$, then $X_{u}^{(r)} = X_{k}^{(r-1)}$ for all $u \not\in \left\{\ell_1,\ell_2,\ell_3,\ell_4\right\}$, and the $4$ new sets are obtained by applying Lemma \ref{BU} in the three different ways to the sets $X_{\ell_1}^{(r-1)}, X_{\ell_2}^{(r-1)}, X_{\ell_3}^{(r-1)}, X_{\ell_4}^{(r-1)}$ from step $r-1$. The subsets $Y_{\ell_{j}} \subset X_{\ell_{j}}^{(r-1)}$ thus obtained for each $j=1,\ldots,4$ are then added to the new list as $X_{\ell_{j}}^{(r)}$. At the end of this process, the collection of sets $\left\{Y_{u} \subset X_u: u=1,\ldots,k\right\}$ from the final list is $2$-separated, by design. Moreover, each $u$ is involved in precisely ${k-1 \choose 3}$ $4$-tuples in ${[k] \choose 4}$, so it is easy to see that each final set $Y_{u}$ satisfies 
$$|Y_u| \geq \frac{|X_u|}{ 2^{3{k-1 \choose 3}}} > \frac{|X_u|}{ 2^{k^3}}.$$ 
\end{proofof}

\parag{Separable sets in convex position.} An important property of $2$-separated sets that we will take advantage of in the proof of Theorem \ref{main} is given by the following.

\begin{prop}\label{lem_tub}
Let $X_1, \ldots, X_k \subset \R^3$ be a $2$-separated collection of sets in convex position. Fix arbitrary $x_i \in X_i$ and consider any plane $H \subset \R^3$ which does not contain any of the points $x_i$. Then there exists another plane $\tilde H \subset \R^3$ such that for any $i \in [k]$ we have $X_i \subset \tilde H^+$ if $x_i \in H^+$ and $X_i \subset \tilde H^-$ if $x_i \in H^-$.
\end{prop}

\begin{proof}
Let $S^+ \subset [k]$ be the set of indices $i$ such that $x_i \in H^+$ and let $S^- = [k] \setminus S^+$. Then the existence of the plane $\tilde H$ satisfying the desired condition is equivalent to showing that 
$$
\conv\left(\bigcup_{i\in S^+} X_i\right) \cap \conv\left(\bigcup_{i\in S^-} X_i\right) = \emptyset.
$$
Suppose that this is not the case and apply Theorem \ref{kirch} to sets $X^+ = \bigcup_{i\in S^+} X_i$ and $X^- = \bigcup_{i\in S^-} X_i$. Let $A^+ \subset X^+$ and $A^- \subset X^-$ be the sets of size $r$ and $5-r$ such that $\conv(A^+) \cap \conv(A^-) \neq \emptyset$. Note that $r \neq 1, 4$ since that would contradict the convex position of the sets $X_1, \ldots, X_k$. 

So we have $r = 2$ or $3$. Let us consider $r=2$, the other case can be obtained by interchanging the roles of $X^-$ and $X^+$. Let $A^+ = \{y_1, y_2\}$ and $A^- = \{y_3, y_4, y_5\}$. For each $j = 1, \ldots, 5$ let $z_j \in \{x_1, \ldots, x_k\}$ be an element such that $y_j$ and $z_j$ belong to the same set $X_{i_j}$ for some $i_j \in [k]$. Let $B^+ = \{z_1, z_2\}$ and $B^- = \{z_3,z_4,z_5\}$. Then the sets $B^+$ and $B^-$ are separated by the plane $H$ and so $\conv(B^+) \cap \conv(B^-) = \emptyset$. By a continuity argument we conclude that one can choose points $w_j$ on the segment $[y_j, z_j]$ such that the segment $[w_1, w_2]$ intersects the boundary of the triangle $\conv(w_3, w_4, w_5)$. By symmetry, we may assume that the intersection point lies on the edge $[w_3, w_4]$. But this implies that 
$$
\conv(X_{i_1} \cup X_{i_2}) \cap \conv(X_{i_3} \cup X_{i_4}) \neq \emptyset,
$$
and since the points $x_{i_1}, x_{i_2}$ and $x_{i_3}, x_{i_4}$ lie on different sides of $H$, we must have $\{i_1, i_2\} \cap \{i_3, i_4\} = \emptyset$. This contradicts the condition that sets $X_1, \ldots, X_k$ are $2$-separated. 
\end{proof}

\section{Proof of Theorem \ref{main}}
	
Fix $\epsilon>0$, and let $X \subset \R^3$ be an arbitrary set in general position with $|X| \geq 2^{\epsilon n}$. We will show that for $n$ sufficiently large, the set $X$ must always contain a convex polytope with $n$ vertices. 

Suppose otherwise, and let $\pi: \R^3 \rightarrow \R^2$ denote a projection along a generic direction. Denote $Y = \pi(X)$. Apply Theorem \ref{lem_pv} to $Y$ with parameter $k_0=n^{1/4}$, and denote by $P$ the supporting convex polygon (a $(k_0+1)$-cap or $(k_0+1)$-cup). Denote by $T_1, \ldots, T_{k_0}$ its support, and let $e_1, \ldots, e_{k_0}$ be the corresponding supporting edges. For each $i=1,\ldots,k_0$, let $Y_i = T_i \cap Y$ be the set of points in $Y$ clustered in the triangular region $T_i$. With this notation, Theorem \ref{lem_pv} states that $|Y_i| > 2^{-40k_0} |X|$ holds for each $i=1,\ldots,k_0$. Last but not least, let us also denote by $X'_i$ the preimage of $Y_i$ in $X$. 

By Theorem \ref{lem_sep}, one can choose subsets $X_i \subset X'_i$ so that the collection $X_1, \ldots, X_{k_0}$ is $2$-separated. We have 
$$
|X_i| \ge 2^{-k_0^3} |X_i'| = 2^{-k_0^3} |Y_i| \ge 2^{-40 k_0 - k_0^3} |X| \ge |X|^{1-\delta}, 
$$
for some $\delta=\delta(\epsilon) \rightarrow 0$ as $n\rightarrow \infty$. 
For every $i = 1, \ldots, k_0$ pick an arbitrary point $x_i \in X_i$ and note that the projections $\pi(x_1), \ldots, \pi(x_{k_0})$ are consecutive vertices of a convex polygon. Let $k$ be the largest number such that $AB(k) < k_0$. Apply Corollary \ref{cor_ab} to the points $x_1, \ldots, x_{k_0}$ and denote the resulting set of indices by $\{i_1, \ldots, i_k\}$. For simplicity let us relabel indices so that $i_1 = 1, \ldots, i_k = k$.

\begin{prop}\label{5_lem}
Let $J = \{j_1 < j_2 < j_3\} \subset [k]$. There exist (unbounded) polytopes $P_J^1$, $P_J^2$ with at most $3$ edges each such that for $j \in [k]$ we have
$$
X_{j} \subset \begin{cases}
P_J^1,\text{ if }j \in [1, j_1) \cup (j_2, j_3],\\ 
P_J^2,\text{ if }j \in [j_1, j_2) \cup (j_3, k],
\end{cases}
$$
and such that sets $P^1_J, P^2_J, \conv(X_{j_2})$ are in convex position. Equivalently, no line $l \subset \R^3$ intersects all 3 sets $P_J^1$, $P_J^2$, $\conv(X_{j_2})$ at once. 
\end{prop}

\begin{proof}
Denote 
\begin{align*}
Z_1 &= X_{1} \cup \ldots \cup X_{j_1-1},\\
Z_2 &= X_{j_1} \cup \ldots \cup X_{j_2-1},\\
Z_3 &= X_{j_2+1} \cup \ldots \cup X_{j_3},\\
Z_4 &= X_{j_3+1} \cup \ldots \cup X_{k}.
\end{align*}
The conclusion of Corollary \ref{cor_ab} implies that the convex hulls of the sets
\begin{align*}
\{x_{1}, \ldots, x_{j_1-1}\} \cup \{x_{j_2+1}, \ldots, x_{j_3}\}\ \ \ \text{and}\ \ \
\{x_{j_1}, \ldots, x_{j_2}\} \cup \{x_{j_3+1}, \ldots, x_{k}\}
\end{align*}
are disjoint, so by Proposition \ref{lem_tub} there must exist a plane $H_1$ which also separates $Z_1 \cup Z_3$ from $Z_2 \cup Z_4 \cup X_{j_2}$. Similarly, the conclusion of Corollary \ref{cor_ab} also implies that the convex hulls of the sets
\begin{align*}
\{x_{1}, \ldots, x_{j_1-1}\} \cup \{x_{j_2},\ldots,x_{j_3}\}\ \ \ \text{and}\ \ \
\{x_{j_1}, \ldots, x_{j_2-1}\} \cup \{x_{j_3+1}, \ldots, x_{k}\}
\end{align*}
are disjoint, so by Proposition \ref{lem_tub} we must also have a plane $H_2$ which separates $Z_1 \cup Z_3 \cup X_{j_2}$ from $Z_2 \cup Z_4$.
Last but not least, let $H_0$ be the plane spanned by the set $\pi^{-1}(e_{j_2})$. Clearly, $H_0$ separates $X_{j_2}$ from $Z_1 \cup Z_2 \cup Z_3 \cup Z_4$.

Choose the orientations of planes $H_0, H_1, H_2$ so that $X_{j_2} \subset H_0^+$, $Z_1 \subset H_1^+$ and $Z_2 \subset H_2^+$.
Now we define
$$
P_J^1 = H_0^- \cap H_1^+ \cap H_2^+,
$$
$$
P_J^2 = H_0^- \cap H_1^- \cap H_2^-.
$$
Then $X_{j_2}$ is separated from $P_J^1 \cup P_J^2$ by the plane $H_0$, and the polytope $P_J^\epsilon$, $\epsilon \in \{1, 2\}$, is separated from $P_J^{3-\epsilon} \cup X_{j_2}$ by the plane $H_\epsilon$.

Clearly, $P_J^\epsilon$ is an intersection of $3$ half-spaces and so it has at most $3$ edges.
\end{proof}

Fix $J = \{j_1 < j_2 < j_3\} \in {[k] \choose 3}$, and define a partial order on $X_{j_2}$ as follows. For $x, x' \in X_{j_2}$ we say that $x \prec_J x'$ if $x \in \conv(\{x'\} \cup P_J^1)$. Observe that an $\prec_J$-antichain is a $P_J^1$-free set and by Proposition \ref{5_lem} a $\prec_J$-chain is a $P_J^2$-free set.
Let us color the triple $J$ red if there is a $P_J^1$-free subset in $X_{j_2}$ of size $|X_{j_2}|^{1/2}$ and blue if there is a $P_J^2$-free subset of size $|X_{j_2}|^{1/2}$.
By Dilworth's theorem \cite{RD50}, note that this represents a well-defined red-blue coloring of the set of triples ${[k] \choose 3}$. 

By Ramsey's theorem \cite{Ramsey}, it follows that for some $t \gg_k 1$, we can find a monochromatic clique $\{j_1, \ldots, j_t\} \subset [k]$. Without loss of generality, let us assume that it is a red clique. Then for any $l \in [t]$, we have
$$
\bigcup_{m:~ m \neq l, l-1} X_{j_m} \subset P^1_{j_{l-1}, j_l, j_t}
$$
and since $\{j_{l-1}, j_l, j_t\}$ is red, we can choose a $P^1_{j_{l-1}, j_l, j_t}$-free set $Z_l \subset X_{j_l}$ of size $|X_{j_l}|^{1/2} \ge |X|^{1/2(1-\delta)}$. If $|Z_l|> {n + 2n/t \choose 2n/t}^3$, then Proposition \ref{lem_es} ensures that $Z_l$ contains either a convex subset of size $n$ or $P^1_{j_{l-1}, j_l, j_t}$-cap of size $2n/t$. 

However, if our original set $X \subset \mathbb{R}^{3}$ does not contain convex sets of size $n$, the former case is automatically impossible for every $l \in [t]$. On the other hand, if each set $Z_{j}$ contains a $P^1_{j_{l-1}, j_l, j_t}$-cap $K_l \subset Z_l$ of size $2n/t$, for every $l \in [t]$, then it is easy to see that $K = K_1 \cup K_3 \cup \ldots \cup K_{2\lceil t/2 \rceil -1}$ is a convex set of size at least $n$.
Indeed, on one hand, for any $l \le t/2$, a point $x \in K_{2j+1}$ can't lie in the convex hull
$$
\conv(P^1_{j_{2l}, j_{2l+1}, j_t} \cup (K_{2j+1} \setminus \{x\})),
$$
whereas, on the other hand, we have $K_{2r+1} \subset P^1_{j_{2l}, j_{2l+1}, j_t}$ for any $r \neq l$; therefore, the point $x$ also does not lie in the convex hull 
$$
\conv\left(\bigcup_{r\neq l} K_{2r+1} \cup (K_{2l+1} \setminus \{x\})\right) = \conv(K \setminus \{x\}).
$$
This implies that the set $K$ is in convex position.

We conclude that if $X$ does not contain a convex set of size $n$ then
$$
|X| \le |Z_l|^{2/(1-\delta)} \le {n + 2n/t \choose 2n/t}^{6/(1-\delta)} \le t^{C n/t},
$$
for some constant $C>0$ and all sufficiently large $n$. Since $t \rightarrow \infty$ as $n \rightarrow \infty$, this implies $|X| = 2^{o(n)}$. This completes the proof of Theorem \ref{main}. 

\smallskip

\remark{ One can verify that our argument produces an $o(n)$ term which is of the form $\frac{n}{\log_{(5)} n}$, where $\log_{(k)}$ denotes the $k$-th iterated logarithm function. As already alluded to in the comment made after Corollary \ref{cor_ab}, several slight optimizations are possible. For example, one can save a log in the upper bound of the above-below function $AB(k)$ by using its semialgebraic nature, and then relying on improved quantitative bounds for semialgebraic Ramsey numbers. Another improvement can come from a closer attention to the last application of Ramsey's theorem in this section; the red/blue coloring of ${[k] \choose 3}$ constructed in the proof of Theorem \ref{main} has an additional monotonicity property: for any set of indices $j_1 \le j_2 < j_3 < j_4 \le j_5$, if the triple $\{j_1, j_3, j_4\}$ is red then the triple $\{j_2, j_3, j_5\}$ is also red. If $R_m(t)$ denotes the smallest $k$ such that any such coloring of ${[k] \choose 3}$ contains a monochromatic clique of size $t$, one can show that $R_m(t)$ is exponential in $t$, which in turn yields a superior dependence between our parameters $t$ and $k$ than the double exponential upper bound on $R_3(n)$ provides. Since all such refinements only get to reduce the number of iterations of the logarithm in the ultimate bound (at the price of substantial technicalities), we decided to not pursue these in any more detail in the current paper in order to maximize the clarity of the main new ideas. 
} 

\section{Proof of Theorem \ref{fr}}

Let $P \subset \R^d$ be a convex polytope. For $i = 0, \ldots, d$ we denote by $\F_i(P)$ the set of faces of $P$ of dimension $i$. For a subset of faces $S \subset \F_{d-1}(P)$ let $R(P, S)$ be the set of points $x \in \R^d$ such that a face $F \in \F_{d-1}(P)$ separates $x$ from $P$ if and only if $F \in S$. For $i = 0, \ldots, d$ we denote $f_i(P) = |\mathcal F_i(P)|$.

\begin{prop}\label{lem_f}
For any simplicial polytope $P \subset \R^d$ and any $t \ge 1$ the number of sets $S \subset \F_{d-1}(P)$ of size $t$ such that $R(P, S)$ is non-empty is at most ${d t \choose t} f_{d-1}(P)$.
\end{prop}

\begin{proof}
Let $G(P)=(V,E)$ be the adjacency graph of $(d-1)$-dimensional faces of $P$. This is a graph with $V = \F_{d-1}(P)$, where two faces $F, F'$ are adjacent if their intersection is a $(d-2)$-dimensional face of $P$. Let $S \subset \F_{d-1}(P)$ be such that $R(P, S)$ is non-empty and contains a point $x \in \R^d$. We claim that then $S$ is connected in $G$.
If $x \in P$ then $S = \emptyset$ is connected. Now suppose that $x \not \in P$ and denote $S= \{F_1, \ldots, F_k\}$, $k\ge 1$. Let $H$ be a hyperplane separating $x$ from $P$. Let $P'$ be the projection of $P$ on the plane $H$ through the point $x$. In other words, $P' = \conv(P \cup \{x\}) \cap H$. 

Let $\phi: P' \rightarrow P$ be a function which maps a point $y \in P'$ to the first point of intersection of the line $(x, y)$ with $P$. We claim that the image of $\phi$ is precisely the union $U = F_1 \cup \ldots \cup F_k$. Indeed, suppose that $F \in \F_{d-1}(P)$ contains the point $\phi(y)$ for some $y \in P'$. Then it is clear that $F$ separates $P$ from $x$ and so $F \in S$, which implies $\phi(y) \in U$. Conversely, if $z \in F_i$ for some $i = 1, \ldots, k$ then $F_i$ separates $P$ from $x$ and so the interval $[x, z]$ does not contain any other points of $P$ and so $z = \phi(y)$, where $y$ is the point of intersection of $H$ with $[x, z]$.

Since the sets $F'_i = \phi^{-1}(F_i)$ form a decomposition (up to polyhedra of smaller dimension) of a convex polytope $P'$ into convex polytopes, the adjacency graph of $F'_i$-s is connected. On the other hand, for any $i,j$ we have $\dim F'_i \cap F'_j = \dim F_i \cap F_j$, which implies that the adjacency graphs of $F_i$-s and $F'_i$-s are isomorphic and so $S$ is connected as well.

The statement now follows from the following graph theoretic fact.
\begin{lemma} \label{Knuth}
Let $G$ be a graph with maximum degree $d$, and let $t\ge 1$. Then the number of connected subgraphs in $G$ of size $t$ is at most $\frac{1}{(d-1)t+1} {d t \choose t} |V(G)|$.
\end{lemma}

\begin{proofof}{Lemma \ref{Knuth}}
Given any vertex $v \in V$, we claim that there are at most $N=\frac{1}{(d-1)t+1} {d t \choose t}$ connected subgraphs in $G$ containing $v$. Indeed, we claim that the maximal number of connected subgraphs of a given size $t$ containing a fixed vertex is attained when $G$ is the infinite $d$-regular tree $T_d$. Let $P_v(G)$ be the set of paths $p$ in $G$ starting from $v$ which do not contain `turning points', i.e. no edge of $G$ appears in $p$ twice in a row. Say that two paths $p, p' \in P_v(G)$ are connected by an edge if one can be obtained from another by adding one edge at the end. Note that this defines a graph on $P_v(G)$, which is in fact an infinite tree with maximum degree $d$. Moreover, note that the assignment 
$$p \mapsto \ \ \text{the end vertex of}\  p\ \text{which is not}\ v$$ defines a graph homomorphism $f: P_v(G) \rightarrow G$. For any connected set $S \subset V(G)$ containing $v$ we can construct a connected subset $S' \subset P_v(G)$ as follows: let $T$ be a spanning tree of $S$ in $G$, then for $x \in S$ consider the unique path $x'$ from $x$ to $v$ in $T$ and let $S'$ be the set of all such paths.

Clearly, we have $f(S') = S$, so each connected set $S \subset V(G)$ defines a unique connected set $S' \subset P_v(G)$ of the same size. Thus, the number of connected subgraphs of size $t$ in $G$ containing $v$ is at most the number of connected subgraphs in $P_v(G)$ containing the empty path. Next, since the maximal degree in $P_v(G)$ is $d$, we can embed it into the infinite $d$-regular tree $T_d$.
Since the number of subtrees in $T_d$ of size $t$ containing a fixed vertex is precisely $N$ (see for example \cite{Knuth}), the conclusion follows.
\end{proofof}

The maximum degree of the adjacency graph of our polytope $P \subset \mathbb{R}^{d}$ is at most $d$ and the sets $S \subset \F_{d-1}(P)$ of size $t$ such that $R(P, S)$ is non-empty induce connected subgraphs of $G$ of size $t$, so Lemma \ref{Knuth} applied for $G(P)$ shows that the number of sets $S \subset \F_{d-1}(P)$ of size $t$ with $R(P, S) \neq \emptyset$ is indeed at most ${d t \choose t} f_{d-1}(P)$.
\end{proof}

For a convex polytope $P$ and a point $x$ in $\R^d$ let $S(P, x) \subset \F_{d-1}(P)$ be the set of faces which separate $x$ from $P$. 

\begin{lemma}\label{o4}
Let $x_1, \ldots, x_k \in \R^d \setminus P$ be such that sets $S(P, x_i)$ are pairwise disjoint. Then points $x_1, \ldots, x_k$ are in convex position. 
\end{lemma}

\begin{proof}
Indeed, take any $F \in S(P, x_i)$, then $x_i$ is separated by $F$ from $P$. But for any $j \neq i$ we have $F \not \in S(P, x_j)$ and so $x_j$ and $P$ are on the same side from $F$. This implies that $F$ separates $x_i$ from all the points $x_j$, $j \neq i$, and so $x_1, \ldots, x_k$ are in convex position.
\end{proof}

\begin{lemma}\label{4c}
Let $X \subset \R^3$ be a finite set in convex position. Then there is a set $Y \subset X$ of size at least $|X|/4$ such that the sets $S(\conv(X\setminus Y), x)$, $x \in Y$, are pairwise disjoint.
\end{lemma}

\begin{proof}
Let $G_{X}$ be the usual graph of the polytope $\conv(X)$. Since $G_{X}$ is planar, the Four Color Theorem states that its vertices $X$ can be colored with $4$ colors in a way such that no edge of $G_{X}$ connects vertices of the same color, or, in other words, that the chromatic number of $G_{X}$, which we denote as usual by $\chi(G_{X})$, is at most $4$ (see for example \cite[Chapter 5]{Diestel} and the references therein). In particular, this implies that the so-called independence number $\alpha(G_{X})$ satisfies the inequality
$$\alpha(G_{X}) \geq \frac{|X|}{\chi(G_{X})} \geq \frac{|X|}{4}.$$
Hence there must exist an independent set $Y \subset X$ of size at least $|X|/4$ in $G_{X}$. We claim that such a set $Y$ has the required property that all the sets in the collection $\left\{S(\conv(X\setminus Y), x):\ x \in Y\right\}$ are pairwise disjoint. Indeed, denote $P = \conv(X \setminus Y)$, and suppose for the sake of contradiction that for some $x, y \in Y$ the sets $S(P, x)$, $S(P, y)$ have a common element $F$.
Let $H$ be the plane containing $F$ such that $P \subset H^+$. Clearly, the set $Z = X \setminus H^+$ is contained in $Y$ and contains at least 2 elements $x, y$. Note that if a half-space contains at least 2 vertices of a polytope then it contains an edge of this polytope. Applying this to the polytope $\conv(X)$ and the open half-space $H^-$ we conclude that $Z$ is not an independent set in the graph of $\conv(X)$. But $Z \subset Y$ and $Y$ is independent, which represents a contradiction.
\end{proof}

Now we can prove Theorem \ref{fr}. Let $\mathcal{X} \subset \R^3$ be a set of size $N \ge ES_3(8n)$ in general position. Then by a standard double counting argument $\mathcal{X}$ contains at least 
$$
\frac{{N \choose 8n}}{{ES_3(8n) \choose 8n}} \geq  \left(\frac{N}{ES_3(8n)}\right)^{8n}
$$
$8n$-element subsets in convex position. Given a convex $8n$-element set $X \subset \mathcal{X}$, apply Lemma \ref{4c} and let $Y_X \subset X$ be the resulting set of size $2n$. Let $Z_X = X \setminus Y_X$. By the pigeonhole principle there is a $6n$-element subset $Z \subset \mathcal{X}$ such that $Z = Z_X$ holds for at least 
$$
{N \choose 6n}^{-1} \left(\frac{N}{ES_3(8n)}\right)^{8n}
$$
convex $8n$-element subsets $X \subset \mathcal{X}$. Denote $P = \conv(Z)$. Now for each $X$ with $Z_X = Z$, consider the following arrangement of $2n$ disjoint sets 
$$
\mathcal S_X = \{S(P, x)~|~ x \in Y_X\}.
$$
By \cite[Section 5, Proposition 5.5.3]{Mat}, note that $P$ has $f_2(P) \le 2 f_0(P) = 12 n$ faces. Fix a sequence of numbers $a_1, \ldots, a_{2n} \ge 1$ such that $a_1 + \ldots + a_{2n} \le 12 n$. By Proposition \ref{lem_f} the number of ways to choose sets $S_1, \ldots, S_{2n}$ such that $|S_i| = a_i$ and $S_i = S(P, x)$ for some $x \in \R^3$ is at most
$$
\prod_{i=1}^{2n} 12 n {3 a_i \choose a_i} \le (12n)^{2n} (3e)^{a_1 + \ldots + a_{2n}} \le 12^{2n} (3e)^{12 n} n^{2n}.
$$
The number of ways to choose the sequence $(a_1, \ldots, a_{2n})$ is ${12n \choose 2n}$. By the pigeonhole principle there exists a collection of sets $\mathcal S = \{S_1, \ldots, S_{2n}\}$ such that $\mathcal S_X = \mathcal S$ for at least
\begin{equation}\label{eql}
{12 n \choose 2n}^{-1} 12^{-2n} (3e)^{-12n} n^{-2n} {N \choose 6n}^{-1} \left(\frac{N}{ES_3(8n)}\right)^{8n} \ge \frac{N^{2n}}{ES_3(8n)^{8n}} \frac{n^{4n}}{C^n} \ge \frac{N^{2n}}{ES_3(8n)^{8n}}
\end{equation}
convex $8n$-element sets $X \subset \mathcal{X}$. Here $C>0$ is an absolute constant and the last inequality holds for sufficiently large $n$.

Now recall that for $x \in \mathcal{X}$ we have $S(P, x) = S_i$ if and only if $x \in R(P, S_i)$. So a set 
$$X = Y \cup \{x_1, \ldots, x_{2n}\}$$ satisfies $\mathcal S_X = \mathcal S$ if and only if after a permutation of indices we have $x_i \in R(P, S_i)$ for $i = 1, \ldots, 2n$. By Lemma \ref{o4} any arrangement of points $x_i \in R(P, S_i)$ is in convex position, so the number of $8n$-element sets $X \subset \mathcal{X}$ in convex position such that $\mathcal S_X = \mathcal S$ is equal to 
$$
|\mathcal{X} \cap R(P, S_1)| \cdot \ldots \cdot |\mathcal{X} \cap R(P, S_{2n})|.
$$
Using (\ref{eql}) and the upper bound $|\mathcal{X} \cap R(P, S_i)| \le N$, we can find indices $i_1, \ldots, i_n \in [2n]$ such that
$$
|\mathcal{X} \cap R(P, S_{i_j})| \ge \frac{N}{ES_3(8n)^8}
$$
holds for any $j = 1, \ldots, n$. Then the sets $X_j = \mathcal{X}\cap R(P, S_{i_j})$ clearly satisfy the statement of the theorem.

\section{Concluding remarks}

In this paper, we proved that $ES_{d}(n) = 2^{o(n)}$ holds for all $d \geq 3$, thus showing that in space and in higher dimensional Euclidean spaces only subexponentially many points are needed in order to ensure the presence of a convex polytope on $n$ vertices. 

The best known lower bound for $ES_{d}(n)$ is due to K\'arolyi and Valtr \cite{KV03}, who showed that there exists a set of $2^{c_{d}n^{\frac{1}{d-1}}}$ points in $\mathbb{R}^{d}$ in general position which contains no convex subset of size $n$, namely 
$$ES_{d}(n) \geq 2^{c_{d}n^{\frac{1}{d-1}}}.$$
Here $c_{d} > 0$ is a constant which depends solely on the dimension $d$. The construction begins with a singleton set $X_{0}$, and $X_{i+1}$ is obtained from $X_i$ by replacing each point $x \in X_i$ with the pair of points
$$x + (\epsilon_i^{d},\epsilon_{i}^{d-1},\ldots,\epsilon_{i})\ \ \text{and}\ \ x - (\epsilon_i^{d},\epsilon_{i}^{d-1},\ldots,\epsilon_{i}),$$
with $\epsilon_i > 0$ sufficiently small, and then perturbing the set slightly so that $X_{i+1}$ remains also in general position. At each step we have $|X_{i}|=2^{i}$, and the main observation is that
$$\operatorname{mc}(X_{i+1}) \leq \operatorname{mc}(X_i) + \operatorname{mc}(\pi(X_i)),$$
where $\operatorname{mc}(X)$ represents the maximum size of a subset of $X$ in convex position, and $\pi$ is the projection to the hyperplane $x_{d}=0$. We would like to conclude this paper by sharing our belief (also an unpublished conjecture of F\"uredi, cf. \cite{KV03}) that this construction may very well be optimal for all $d \geq 3$, apart from the precise value of the constant $c_d$ in the exponent.

\bigskip

{\bf{Acknowledgements}}. We would like to thank Karim Adiprasito, Boris Bukh, and David Conlon for helpful discussions.

\footnotesize{

\textsc{School of Mathematics, Institute for Advanced Study, Princeton, NJ 08540, USA}


{\it{Email address}}: \href{mailto:cosmin.pohoata@gmail.com}{\nolinkurl{cosmin.pohoata@gmail.com}} 

\bigskip


\textsc{Department of Mathematics, Massachusetts Institute of Technology, Cambridge, MA 02139, USA}

{\it{Email address}}: \href{mailto:zakharov2k@gmail.com}{\nolinkurl{zakharov2k@gmail.com}}
}

\end{document}